\documentclass[12pt,a4paper,english,reqno]{amsart}
\usepackage[a4paper,footskip=1.5em]{geometry}
\usepackage{amsmath,amssymb,amsthm,mathtools,bbm}
\usepackage[mathscr]{euscript}
\usepackage[usenames,dvipsnames]{color}
\usepackage{adjustbox,tikz,calc,graphics,babel,standalone}
\usepackage{subcaption}
\usepackage{csquotes,enumerate,verbatim}
\usepackage[final]{microtype}
\usepackage[numbers]{natbib}
\usetikzlibrary{shapes.misc,calc,intersections,patterns,decorations.pathreplacing}
\usetikzlibrary{arrows,shapes,positioning}
\usetikzlibrary{decorations.markings}
\usepackage{hyperref}
\hypersetup{colorlinks=true,linkcolor=blue,citecolor=blue,pdfpagemode=UseNone,pdfstartview={XYZ null null 1.00}}
\usepackage{cmtiup}

\tikzstyle arrowstyle=[scale=1]

\theoremstyle{plain}
\newtheorem*{theorem*}{Theorem}
\newtheorem{theorem}{Theorem}[section]
\newtheorem{lemma}[theorem]{Lemma}
\newtheorem{claim}[theorem]{Claim}
\newtheorem{proposition}[theorem]{Proposition}
\newtheorem*{claim*}{Claim}

\newtheorem{conjecture}[theorem]{Conjecture}
\newtheorem{problem}[theorem]{Problem}

\theoremstyle{definition}

\theoremstyle{remark}

\def\PP{\mathcal{P}}
\def\QQ{\mathcal{Q}}
\def\CC{\mathscr{C}}
\def\DD{\mathscr{D}}
\def\EE{\mathscr{E}}
\def\FF{\mathscr{F}}
\def\N{\mathbb{N}}

\def\cycprec{\prec}

\let\originalleft\left
\let\originalright\right
\renewcommand{\left}{\mathopen{}\mathclose\bgroup\originalleft}
\renewcommand{\right}{\aftergroup\egroup\originalright}

\makeatletter
\def\imod#1{\allowbreak\mkern10mu({\operator@font mod}\,\,#1)}
\makeatother

\begin{document}

\title{Long cycles in Hamiltonian graphs}

\author{Ant\'onio Gir\~ao}
\address{Department of Pure Mathematics and Mathematical Statistics, University of Cambridge, Wilberforce Road, Cambridge CB3\thinspace0WB, UK}
\email{A.Girao@dpmms.cam.ac.uk}

\author{Teeradej Kittipassorn}
\address{Departamento de Matem\'atica, Pontif\'icia Universidade Cat\'olica do Rio de Janeiro (PUC-Rio), Rua Marqu\^es de S\~ao Vicente 225, G\'avea, Rio de Janeiro, RJ 22451-900, Brazil}
\email{ping41@mat.puc-rio.br}

\author{Bhargav Narayanan}
\address{Department of Mathematics, Rutgers University, Piscataway NJ 08854, USA}
\email{narayanan@math.rutgers.edu}

\date{12 September 2017}

\begin{abstract}
We prove that if an $n$-vertex graph with minimum degree at least $3$ contains a Hamiltonian cycle, then it contains another cycle of length $n-o(n)$; this implies, in particular, that a well-known conjecture of Sheehan from 1975 holds asymptotically. Our methods, which combine constructive, poset-based techniques and non-constructive, parity-based arguments, may be of independent interest.
\end{abstract}

\maketitle

\section{Introduction}
A \emph{Hamiltonian cycle} in a graph $G$ is a cycle spanning the vertex set of $G$, and a graph is said to be \emph{Hamiltonian} if it contains a Hamiltonian cycle. Over the last seventy years, the following problem has received a great deal of attention: under what conditions does a graph $G$ with a Hamiltonian cycle $\CC$ contain another long cycle distinct from $\CC$? Of course, for this question to be interesting, one needs to ensure that $G$ contains additional edges (not already in $\CC$); a moment's thought further reveals that additional edges are not enough in and of themselves, but rather, one requires additional edges that are `equidistributed' over the vertex set of $G$. This problem of understanding when the presence of additional edges in a Hamiltonian graph forces the existence of another long (possibly Hamiltonian) cycle has a storied history; see the surveys of Gould~\citep{survey} and Bondy~\citep{prob_sur} for an overview. 

Our main contribution here is to show that perhaps the weakest possible condition promising some form of `equidistribution of additional edges' in a graph with a Hamiltonian cycle is sufficient to guarantee the existence of another long cycle; writing $\delta(G)$ for the minimum degree of a graph $G$, we prove the following.
\begin{theorem}\label{mainthm}
For all $n \in \N$, if an $n$-vertex graph $G$ with $\delta(G) \ge 3$ contains a Hamiltonian cycle, then $G$ contains another cycle of length at least $n-cn^{4/5}$, where $c >0$ is an absolute constant.
\end{theorem}
To provide some context for Theorem~\ref{mainthm}, we remind the reader of the most famous open problem in the area; the following long outstanding conjecture is due to Sheehan~\citep{sheehan}.
\begin{conjecture}\label{conj:shee}
For each integer $d\ge3$, every $d$-regular Hamiltonian graph contains a second Hamiltonian cycle. 
\end{conjecture}
Conjecture~\ref{conj:shee} was proposed as an extension of the classical result of Smith, see~\citep{tutte}, that establishes the above conjecture in the case where $d = 3$. Sheehan's conjecture was subsequently shown to hold for all odd $d\ge 3$ by Thomason~\citep{thomason} using a beautiful, non-constructive, parity-based argument, and for all $d\ge 300$ by Thomassen~\citep{thomassen1, thomassen2} using an ingenious combination of Thomason's argument and the Lov\'asz local lemma. We refer the reader to the paper of Haxell, Seamone and Verstra\"ete~\citep{haxell} for both the current state of the art as well as a discussion of why existing methods are unlikely to settle Conjecture~\ref{conj:shee} in its full generality.

In the light of Sheehan's conjecture, it is natural to ask if regularity is genuinely necessary to force the existence of a second Hamiltonian cycle, or if a weaker condition on the minimum degree, say, might suffice instead. In particular, the following question suggests itself: does every Hamiltonian graph $G$ with $\delta(G) \ge 3$ contain a second Hamiltonian cycle? Entringer and Swart~\citep{unique} answered this question negatively by constructing infinitely many Hamiltonian graphs without a second Hamiltoninan cycle, all with minimum degree $3$. While the Hamiltonian graphs with minimum degree $3$ constructed by Entringer and Swart only contain a single Hamiltoninan cycle each, these graphs do contain other long cycles that almost span the entire vertex set;  it is therefore natural to ask if such a situation is unavoidable in general.

\begin{problem}\label{mainconj}
If an $n$-vertex graph $G$ with $\delta(G) \ge 3$ contains a Hamiltonian cycle, then must $G$ contain another cycle of length $n-o(n)$?
\end{problem}

Of course, Problem~\ref{mainconj} is closely related to Conjecture~\ref{conj:shee} since an affirmative answer to the above question would assert precisely that an asymptotic form of Sheehan's conjecture holds under significantly milder degree conditions than the regularity restrictions prescribed in Conjecture~\ref{conj:shee}; our main result furnishes, in a quantitative form, precisely such an affirmative answer.

Perhaps the most interesting aspect of Theorem~\ref{mainthm} is the fact that its proof is based on a combination of constructive and non-constructive arguments: to prove our main result, we use poset-based techniques and parity-based arguments in conjunction with each other, so our methods might be of independent interest.
 
This paper is organised as follows. We first introduce some notation and collect together the tools that we need for the proof of our main result in Section~\ref{s:pre}. We then prove Theorem~\ref{mainthm} in Section~\ref{s:proof}. Finally, we conclude in Section~\ref{s:conc} with a discussion of some open problems.

\section{Preliminaries}\label{s:pre}
It will be convenient to begin by establishing some notation for dealing with Hamiltonian graphs.

Given a graph $G$ with a designated Hamiltonian cycle $\CC$, we shall always fix one of the two possible cyclic orderings of $V(G)$ obtained be traversing $\CC$ to be canonical. Therefore, when we speak, for example, about following $\CC$ from $x$ to $y$ for $x,y \in V(G)$, we mean this with respect to the canonical ordering. We use $\cycprec$ to specify relative positions with respect to the canonical ordering, so for instance, given $x,y,z \in V(G)$, we write $x \cycprec y \cycprec z$ (or equivalently either $y \cycprec z \cycprec x$ or $z \cycprec x \cycprec y$) to mean that we encounter $x$, $y$ and $z$ in that order around $\CC$. Finally, for $x, y \in V(G)$, we write $d_\CC (x,y)$ for the length of the path from $x$ to $y$ around $\CC$ following the canonical ordering, noting that $d_\CC (x,y) \ne d_\CC (y,x) $ in general.

Let $G$ be a graph with a designated Hamiltonian cycle $\CC$. Any cycle of $G$ distinct from $\CC$ is said to be \emph{nontrivial}. We call any edge of $G$ not in $\CC$ a \emph{chord}. Observe that there exist two subsets of the vertex set of $G$ corresponding to each chord $e$ of $G$, namely the vertex sets of the two paths traversing $\CC$ between the endpoints of $e$; we call these two sets of vertices the two \emph{domains} of $e$, and note that the domains of $e$ intersect precisely in the endpoints of $e$. We say that a chord $e$ is \emph{minimal} if at least one of its domains induces no chords of $G$ other than $e$ itself, and we call the corresponding domain of $e$ its \emph{minimal domain}; here, if both domains of $e$ induce no chords, then we arbitrarily choose one these domains to be the minimal domain of $e$. We say that a pair of chords \emph{interlace} if their endpoints are all distinct and appear in alternating order around $\CC$ (in the canonical ordering of the vertex set, say); otherwise, we say that they are \emph{parallel}. Also, we say that a set of chords is \emph{independent} if no two of the chords in the set share an endpoint. Finally, we say that two vertices $x,y \in V(G)$ are \emph{chord-adjacent} if they are connected by a chord of $G$.

Next, we collect together some tools that we shall require for the proof of our main result. 

To handle the constructive half of our argument, we shall require a well-known consequence of a classical result of Dilworth~\citep{dill}. Recall that in a partially ordered set (or poset for short), a \emph{chain} is a subset in which each pair of elements is comparable (which makes a chain a linearly ordered set), and an \emph{antichain} is a subset in which no two elements are comparable; we need the following fact.

\begin{proposition}\label{dilworth}
For $r,s \in \N$, every poset of size $rs$ contains either a chain of size $r$ or an antichain of size $s$. \qed
\end{proposition}

The non-constructive half of our argument depends on the following convenient formulation, due to Thomassen~\citep{thomassen3}, of the parity-based `lollipop argument' of Thomason~\citep{thomason}. Recall that a set $X$ of vertices \emph{dominates} another set $Y$ of vertices and edges in a graph if each vertex in $Y$ is adjacent to some vertex in $X$ and if each edge in $Y$ is incident to some vertex in $X$.

\begin{proposition}\label{lollipop}
Let $G$ be a graph with a designated Hamiltonian cycle $\CC$. If there exists a set $X \subset V(G)$ such that
\begin{enumerate}
\item $X$ is independent in the graph $G' = (V(G), E(\CC))$, and
\item $X$ dominates $V(G) \setminus X$ in the graph $G'' = (V(G), E(G) \setminus E(\CC))$,
\end{enumerate}
then $G$ contains a nontrivial Hamiltonian cycle. \qed
\end{proposition}

Finally, we use standard asymptotic notation throughout to suppress absolute constants, and for the sake of clarity of presentation, we systematically omit floor and ceiling signs whenever they are not crucial.

\section{Proof of the main result}\label{s:proof}
We begin with the following lemma that allows us to handle Hamiltonian graphs with many interlacing chords.

\begin{lemma}\label{manyinter}
Let $G$ be an $n$-vertex graph with a designated Hamiltonian cycle $\CC$. If $G$ contains a set $I$ of $2m$ independent chords made up of $m$ interlacing pairs for some $m \ge 1$, then $G$ contains a nontrivial cycle missing $O(n/m^{1/3})$ vertices.
\end{lemma}
\begin{proof} 
Note that if $G$ has at least one chord, then $G$ contains a nontrivial cycle. In what follows, we therefore suppose, as we may, that $m$ is sufficiently large.
We shall show, assuming $m$ is suitably large, that it is possible to  construct a cycle of the required length using at most $4$ chords of $G$ and the edges of $\CC$.
	
We begin by constructing two posets on any set $S$ of independent chords in $G$ as follows. We fix some edge $f$ of $\CC$, and for a chord $e$ of $G$, we call the domain of $e$ containing the endpoints of $f$ the \emph{interior} of $e$, and the other domain the \emph{exterior} of $e$. We then define a partial order $\PP_S$ on $S$ by saying $e_1 < e_2$  for $e_1 , e_2 \in S$ if the interior of $e_1$ is contained in the interior of $e_2$. Next, we fix a linear order $\mathcal{L}$ of the vertices of $G$ by starting at one of the endpoints of $f$ and following $\CC$ to the other endpoint of $f$, and then define another poset $\QQ_S$ on $S$ by saying that $e_1 < e_2$ for $e_1, e_2 \in S$ if both the endpoints of $e_1$ precede both the endpoints of $e_2$ in $\mathcal{L}$.

The following observation guarantees the existence of a large set of chords with useful structural properties.

\begin{claim}\label{poset}
For any $K>0$, given a set $S$ of independent chords in $G$ of size $Km$, we may find either
\begin{enumerate}
\item a chain in $\PP_S$ of size $Km^{1/3}$,
\item a chain in $\QQ_S$ of size $m^{1/3}$, or
\item an antichain in both $\PP_S$ and $\QQ_S$ of size $m^{1/3}$.
\end{enumerate}
Moreover, in either of the latter two cases, we may find a nontrivial cycle of length at least $n-n/m^{1/3}$ in $G$.
\end{claim}
\begin{proof}
By Proposition~\ref{dilworth}, we see that $\PP_S$ contains either a chain of size $Km^{1/3}$ or an antichain of size $m^{2/3}$. Applying Proposition~\ref{dilworth} again to such an antichain if it exists, we see that either $\QQ_S$ contains a chain of size $m^{1/3}$, or there exists an antichain in both $\PP_S$ and $\QQ_S$ of size $m^{1/3}$. 

If $\QQ_S$ contains a chain of size $m^{1/3}$, then it is easy to see that this chain contains a chord whose exterior contains at most $n/m^{1/3}$ vertices, in which case we are done.

If there exists an antichain in both $\PP_S$ and $\QQ_S$ of size $m^{1/3}$, then it is clear that this antichain consists of pairwise interlacing chords. We may then find, using the pigeonhole principle, chords $uv$ and $xy$ in this antichain with $u \cycprec x \cycprec v \cycprec y$ such that $d_\CC (u, x) + d_\CC (v, y) \le n/m^{1/3}$, in which case we are again done.
\end{proof}

For the rest of the proof, we restrict our attention to the set $I$ and the poset $\PP = \PP_I$; in what follows, any ordering of chords in $I$ will implicitly mean their ordering in $\PP$.
Furthermore, we may assume going forwards that in any set $S \subset I$ of size at least $m/8$, there exists a chain in $\PP$ of size at least $m^{1/3}/8$; indeed, we are done by Claim~\ref{poset} if this is not the case.

We say that a triple $\{u_1v_1<u_2v_2<u_3v_3\}$ of independent chords in $I$ with $u_1\cycprec u_2\cycprec u_3 \cycprec v_3\cycprec v_2\cycprec v_1$ is \emph{tight} if 
\[d_\CC (u_1, u_3) + d_\CC (v_3, v_1) \le 24n/m^{1/3}.\] 
This definition of a tight triple is motivated by the following observation.

\begin{claim}\label{interlacingmiddle}
If $G$ contains two tight triples whose middle chords interlace, then $G$ contains a nontrivial cycle of length at least $n-48n/m^{1/3}$. 
\end{claim} 
\begin{proof}
This claim follows from a somewhat tedious analysis of a few different cases; this analysis requires us to establish some notation first. For a tight triple $U = \{u_1v_1<u_2v_2<u_3v_3\}$ with $u_1\cycprec u_2\cycprec u_3 \cycprec v_3\cycprec v_2\cycprec v_1$, we say that a vertex \emph{lies inside the strip} of $U$ if it lies either on the path $P(u_1, u_3)$ between $u_1$ and $u_3$ in $\CC$ containing $u_2$, or on the path $P(v_3, v_1)$ between $v_3$ and $v_1$ in $\CC$ containing $v_2$.

Suppose that $T_1=\{u_1v_1<u_2v_2<u_3v_3\}$ with $u_1\cycprec u_2\cycprec u_3 \cycprec v_3\cycprec v_2\cycprec v_1$ and $T_2=\{x_1y_1<x_2y_2<x_3y_3\}$ with $x_1\cycprec x_2\cycprec x_3 \cycprec y_3\cycprec y_2\cycprec y_1$ are two tight triples whose middle chords $u_2v_2$ and $x_2y_2$ interlace. 

Assume first that $T_1$ and $T_2$ are not disjoint, and say $u_1v_1=x_1y_1$ with $u_1=x_1$ and $v_1=y_1$. Suppose, as we may, that $u_1\cycprec x_2\cycprec u_2$; we then obtain a cycle using the chords $u_2v_2$ and $x_2y_2$ missing at most 
\[d_\CC (u_1, u_3) + d_\CC (y_3, y_1) \le 48n/m^{1/3}\] 
vertices of $G$, as required.

Therefore, we may suppose that $T_1$ and $T_2$ are disjoint. Suppose first that $x_2$ and $y_2$ lie inside the strip of $T_1$. If both $x_2$ and $y_2$ lie on $P(u_1, u_3)$, then we obtain a cycle using just the chord $x_2y_2$ missing at most $d_\CC (u_1, u_3) \le 24n/m^{1/3}$ vertices. If $x_2$ lies on $P(u_1, u_3)$ and $y_2$ lies on $P(v_3, v_1)$ on the other hand, then we obtain a cycle using the  chords $u_2v_2$ and $x_2y_2$ missing at most 
\[d_\CC (u_1, u_3) + d_\CC (v_3, v_1) \le 24n/m^{1/3}\] 
vertices of $G$.

Therefore, suppose that $x_2$ lies outside the strip of $T_1$ and that $u_2$ lies outside the strip of $T_2$. Suppose without any loss of generality that  $u_2\cycprec u_3\cycprec x_2\cycprec v_3\cycprec v_2$ and $y_2\cycprec y_1\cycprec u_2\cycprec x_1\cycprec x_2$, so either $u_2\cycprec u_3\cycprec x_1\cycprec x_2$ or $u_2\cycprec x_1\cycprec u_3\cycprec x_2$. 

First, suppose that $u_2\cycprec u_3\cycprec x_1\cycprec x_2$, in which case, both $u_2v_2$ and $u_3v_3$ interlace with both $x_1y_1$ and $x_2y_2$. We may then obtain a cycle using the chords $u_2v_2$, $u_3v_3$, $x_1y_1$ and $x_2y_2$ missing at most 
\[d_\CC (u_1, u_3) + d_\CC (v_3, v_1) + d_\CC (x_1, x_3) + d_\CC (y_3, y_1) \le 48n/m^{1/3}\]
vertices of $G$.

Now, suppose that $u_2\cycprec x_1\cycprec u_3\cycprec x_2$. If $u_3\cycprec v_3\cycprec x_3$, then we obtain a cycle using the chord $u_3v_3$ missing at most $d_\CC (x_1, x_3) \le 24n/m^{1/3}$ vertices. Therefore, suppose that $u_3\cycprec x_3\cycprec v_3$. If $y_3\cycprec v_2\cycprec y_2$, then we obtain a cycle using the chords $u_2v_2$ and $x_2y_2$ missing at most 
\[d_\CC (u_2, x_2) + d_\CC (v_2, y_2) \le d_\CC (u_1, u_3) + d_\CC (x_1, x_3) + d_\CC (y_3, y_1) \le 48n/m^{1/3}\]
vertices. Hence, suppose that $v_2\cycprec y_3\cycprec y_2$, so that both $u_2v_2$ and $u_3v_3$ interlace with both $x_2y_2$ and $x_3y_3$. In this case, we obtain a cycle using the chords $u_2v_2$, $u_3v_3$, $x_2y_2$ and $x_3y_3$ missing at most 
\[d_\CC (u_1, u_3) + d_\CC (v_3, v_1) + d_\CC (x_1, x_3) + d_\CC (y_3, y_1) \le 48n/m^{1/3}\]
vertices of $G$.
\end{proof}

Continuing the proof of Lemma~\ref{manyinter}, recall our assumption that in any set $S \subset I$ of size at least $m/8$, there exists a chain in $\PP$ of size at least $m^{1/3}/8$. This assumption implies that there are many pairwise disjoint tight triples in $I$, as we demonstrate below.

\begin{claim}\label{closetriple}
For $K\ge 1/2$, any set $S \subset I$ of size $Km$ contains $Km/4$ pairwise disjoint tight triples. 
\end{claim}
\begin{proof}
We shall show that given any collection $T$ of at most $Km/4$ pairwise disjoint tight triples from $S$, we may find a tight triple from the remaining chords in $S$ which is pairwise disjoint from each of the tight triples in $T$. We know that $S$ contains a subset $S'$ of at least $Km - 3Km/4 \ge Km/4 \ge m/8$ chords none of which appear in any of the triples in $T$. By our assumption, we know that $S'$ contains a chain $u_1v_1<u_2v_2<\dots<u_kv_k$ of size $k = m^{1/3}/8 \ge 6$ in $\PP$ with $u_1\cycprec u_2\cycprec \dots \cycprec u_k\cycprec v_k\cycprec v_{k-1}\cycprec \dots \cycprec v_1$. By considering a partition of $\CC$ into paths with endpoints in $\{u_1,u_3,\dots,v_1,v_3,\dots\}$, we have
\[\sum_{i=1}^{\lceil k/2 \rceil-1} \left(d_\CC (u_{2i-1}, u_{2i+1}) + d_\CC (v_{2i+1}, v_{2i-1})\right) \le n,\]
so there exists an index $1 \le i \le \lceil k/2 \rceil-1$ such that
\[d_\CC (u_{2i-1}, u_{2i+1}) + d_\CC (v_{2i+1}, v_{2i-1}) \le \frac{n}{k/2 - 1} \le \frac{3n}{k} = \frac{24n}{m^{1/3}};\]
this implies that the triple $\{u_{2i-1}v_{2i-1} < u_{2i}v_{2i} < u_{2i+1}v_{2i+1}\}$ is tight, proving the claim.
\end{proof}

We may now finish the proof of Lemma~\ref{manyinter} as follows. By Claims~\ref{interlacingmiddle} and~\ref{closetriple}, we see that $I$ contains $m/2$ pairwise disjoint tight triples whose middle chords are all parallel and independent. Applying Claim~\ref{closetriple} again to the $m/2$ interlacing partners of the middle chords of the triples above, we obtain $m/8$ new pairwise disjoint tight triples; in particular, there exist two tight triples whose middle chords interlace, so we are done by Claim~\ref{interlacingmiddle}.
\end{proof}

In order to handle Hamiltonian graphs with many parallel chords, we shall rely on the non-constructive argument implicit in Lemma~\ref{lollipop}. In order to apply this lemma in the proof of our main result, we shall require a fair bit of preparation; this is accomplished in the somewhat technical lemma that follows below.

\begin{lemma}\label{manypar}
Let $G$ be an $n$-vertex graph with a designated Hamiltonian cycle $\CC$ with the property that no two chords of $G$ interlace. Suppose that no vertex of $G$ is chord-adjacent to two consecutive vertices of $\CC$, and that no two vertices of $G$ of degree greater than $3$ are chord-adjacent. Also, assume that there are subsets $R$ and $B$ of $V(G)$ (whose elements we shall call red and blue respectively) such that 
\begin{enumerate}
\item every vertex in $R\cup B$ has degree $3$, and
\item no two vertices in $R\cup B$ are chord-adjacent.
\end{enumerate}
Then, writing $M\geq 2$ for the number of minimal chords in $G$ and setting $r=|R|$, there exists a set $S\subseteq V(G)$ of vertices such that
\begin{enumerate}
\item $S$ dominates the chords of $G$,
\item $S$ contains no red vertices, and
\item $S$ contains at most $r+M-2$ pairs of consecutive vertices of $\CC$, and none of these pairs contains a blue vertex.
\end{enumerate}
\end{lemma}

\begin{proof}[Proof of Lemma~\ref{manypar}]
We prove this lemma by induction on the number of minimal chords as follows.

First, we prove the base case. Suppose that $G$ has exactly two minimal chords. Let $e=xy$ and $f = uv$ be the two minimal chords, and since $e$ and $f$ cannot interlace by assumption, we may assume that $x \cycprec u \cycprec v \cycprec y$. We say that a vertex is \emph{upstairs} if it lies between $x$ and $u$ on $\CC$, and \emph{downstairs} if it lies between $v$ and $y$ on $\CC$; we write $U$ and $D$ for the sets of vertices upstairs and downstairs respectively. Note that $E(G)\setminus E(\CC)$ is a collection of stars, each of which is such that its centre is upstairs and all of its leaves are downstairs, or vice versa; let these stars be $S_1, S_2, \dots ,S_k$. Note that the centres of these stars are necessarily uncoloured; we adopt the convention that the centre of a trivial star consisting of a single edge is one of its uncoloured vertices. Furthermore, these stars come with a natural ordering: for $i<j$, all the vertices of $S_i$ upstairs are closer to $x$ than all the vertices of $S_j$ upstairs, and all the vertices of $S_i$ downstairs are closer to $y$ than all the vertices of $S_j$ downstairs. To ensure that $S$ dominates the chords of $G$, we shall construct $S$ by choosing, for each $1 \le i \le k$, either to add all the vertices of $S_i$ that are upstairs to $S$, or to add all the vertices of $S_i$ that are downstairs to $S$. Since no pair of leaves of any of these stars are consecutive vertices of $\CC$, $S$ can contain a pair of consecutive vertices of $\CC$ only if the pair spans two stars. Without loss of generality, we may assume that there are $r$ stars containing a red vertex and denote them by $S_{i_1},S_{i_2},\dots,S_{i_r}$. We partition the set of all stars into $r+1$ blocks as
\[\{S_{i_0},\dots,S_{i_1-1}\} \cup \{S_{i_1},\dots,S_{i_2-1}\} \cup \dots \cup \{S_{i_{r-1}},\dots,S_{i_r-1}\} \cup \{S_{i_{r}},\dots,S_{i_{r+1}-1}\},\]
where $i_0=1$ and $i_{r+1}=k+1$. For each $0 \le j \le r-1$, we shall pick vertices in the block $\{S_{i_j},\dots,S_{i_{j+1}-1}\}$ ensuring that the last vertex picked is not blue, and that we pick at most one pair of consecutive vertices of $\CC$ from $\{S_{i_j},\dots,S_{i_{j+1}}\}$. In the case where $j = r$, we shall ensure that we create no pair of consecutive vertices of $\CC$ from the last block. 

For $0 \le j \le r$, we handle the corresponding block of stars as follows. Without loss of generality, suppose that there is a red vertex downstairs in $S_{i_j}$, and consider the sequence
\[S_{i_j} \cap U,S_{i_j+1} \cap D,S_{i_j+2} \cap U,\dots\] 
of candidates for addition to $S$, where the sequence above goes up to the star with the index ${i_{j+1}-1}$. We enlarge $S$ using the block under consideration as follows. If $j=r$, then we add all the vertices in the sequence above. If $j < r$ and the last element in the sequence above containing vertices of $S_{i_{j+1}-1}$ is on the same side (upstairs or downstairs) as a red vertex of $S_{i_{j+1}}$, then we again add all the vertices in the sequence above. Suppose now that $j < r$ and that the last element in the sequence containing vertices of $S_{i_{j+1}-1}$ is on the opposite side as a red vertex of $S_{i_{j+1}}$. Let ${i_j+t}$ denote the index of the last set in the above sequence that does not contain a blue vertex, and note that $t \ge 0$. In this case, we add all the vertices in the sequence above up to the index ${i_{j}+t}$, and then add all the vertices in the complementary sequence (obtained by selecting vertices on the opposite side) from the index ${i_{j}+t+1}$ to the index ${i_{j+1}-1}$. It is clear from the properties that $G$ is assumed to have that this selection procedure generates at most one pair of consecutive vertices of $\CC$ (possibly between $S_{i_j+t}$ and $S_{i_j+t+1}$) from this block, and it is also clear that the last vertex added to $S$ from this block is not blue. Note that in the case where $j=0$, if the corresponding block is nonempty, then there are no red vertices in this block; therefore, we can ensure that when considering the first nonempty block (which corresponds to either $j= 0$ or $j=1$), the first set in the sequence above contains the centre but not the leaves of the first star in the block; we shall need this additional property later in the induction step.

It is easy to check that the above procedure applied to each of the $r+1$ blocks of stars produces a set $S$ as required, proving the base case of the induction.

Next, suppose that $M \ge 3$. Pick a minimal chord $f$. Among all chords whose domain inducing $f$ induces no other chords (except the chord in question itself), pick a chord $e=xy$ which is maximal with respect to the order of its domain inducing $f$; denote the domains of $e$ by $A$ and $B$, where $A$ is the domain of $e$ inducing $f$. Clearly, both $G[A]$ and $G[B]$ are Hamiltonian graphs satisfying the conditions of the lemma; moreover, $G[A]$ has at most $2$ minimal chords, and by our maximal choice of $e$, it is also clear that $G[B]$ has exactly $M-1$ minimal chords. 

We now apply the inductive hypothesis to the graphs $G_A$ and $G_B$ that we now define. First, $G_A$ is obtained from $G[A]$ by adding a new uncoloured vertex $z$ and joining it to $x$ and $y$. It is clear that $G_A$ has at most two minimal chords; say $G_A$ contains $r_1$ red vertices, and set $r_2 = r - r_1$. Next, we obtain $G_B$ from $G[B]$ by recolouring some vertices as follows. Without loss of generality, we may assume that $y$ is the uncoloured centre of the star containing $e$ in $E(G)\setminus E(\CC)$. Let $w$ be the neighbour of $y$ in $\CC$ that belongs to $G[B]$. We make $w$ red in $G_B$ if it was coloured blue in $G$ (and do not alter its colour otherwise), and if $x$ was red or blue in $G$, then we make $x$ an uncoloured vertex in $G_B$. Clearly, $G_B$ has $M-1$ minimal chords, and either at most $r_2+1$ or at most $r_2$ red vertices depending on whether or not the colour of $w$ was altered in $G_B$. 

Let $S_A$ and $S_B$ be the sets obtained inductively in $G_A$ and $G_B$ respectively. First, $e = xy$ is a minimal chord in $G_A$, and $G_A$ has at most two minimal chords, so we can ask for $S_A$ to contain $y$ but not $x$ by arguing as in the base case earlier. Next, note that $S_B$ either contains at most $(r_2+1)+(M-1)-2$ pairs of consecutive vertices  of $\CC$, or at most $r_2+(M-1)-2$ pairs of consecutive vertices of $\CC$, depending on whether or not we had to alter the colour of $w$ in $G_B$. Also, observe that $x$ has degree $2$ in $G_B$, so we may assume that $S_B$ does not contain $x$. 

We now claim that $S=S_A\cup S_B$ is sufficient for our purposes. It is clear that $S$ dominates $E(G)\setminus E(\CC)$ and contains no red vertices of $G$. It is also clear, by induction, that $S$ does not contain a consecutive pair of $\CC$ in which one of the vertices is coloured blue in $G$. Next, if the colour of $w$ was altered in $G_B$, then $S$ does not contain any consecutive pairs of $\CC$ spanning $S_A$ and $S_B$ since $x \not\in S_A \cup S_B$ and $w\not\in S_B$, and if not, then $S$ contains at most one such pair (namely, the edge $yw$); the number of pairs of consecutive vertices of $\CC$ in $S$ is therefore is at most \[(r_2+1) +(M-1)-2 + r_1=r+M-2\]
in the former case, or at most 
\[r_2+(M-1)-2+r_1+1 = r+M-2\] 
in the latter case, as required.
\end{proof}

Armed with Lemmas~\ref{manyinter} and~\ref{manypar}, we are now in a position to prove our main result.

\begin{proof}[Proof of Theorem~\ref{mainthm}]
Let $G$ be an $n$-vertex graph with a designated Hamiltonian cycle $\CC$. We assume, without loss of generality, that $G$ is minimal in the sense that no two vertices with degree greater than $3$ in $G$ are chord-adjacent.

Let $2m$ be the maximum size of a set $I$ of independent chords in $G$ which may be partitioned into $m$ interlacing pairs. If $m \ge n^{3/5}$, then the result follows from Lemma~\ref{manyinter}, so we may suppose that $m \le n^{3/5}$.

Let $P$ denote the set of $4m$ endpoints of the chords in $I$, and consider the graph $G'$ on the same vertex set as $G$ obtained by deleting every chord of $G$ incident to some vertex in $P$; of course, $G'$ is also an $n$-vertex graph in which $\CC$ is the designated Hamiltonian cycle, and from the maximality of $I$, we see that no two chords of $G'$ interlace. We now transform $G'$ as follows: if $x$ and $y$ are consecutive vertices of $\CC$ that are both chord-adjacent to some vertex of $G'$, then we contract the edge $xy$ of $\CC$, and repeat this operation until it is no longer possible to do so. Let $H$ be the resulting graph, and let $\DD$ be its designated Hamiltonian cycle obtained from $\CC$ after these contractions; note that our contractions ensure that no vertex of $H$ is chord-adjacent to two consecutive vertices of $\DD$.

Now, the set of minimal chords of $G'$ with respect to $\CC$ is the same (up to the obvious identification) as the set of minimal chords of $H$ with respect to $\DD$, and furthermore, the size of the minimal domains of these minimal chords are identical in both $G'$ and $H$. Moreover, it is easy to see that $H$ does not contain a pair of interlacing chords. We call any vertex of $H$ that corresponds to one or more contracted edges of $G'$ a \emph{contracted vertex}, and we colour a contracted vertex red in $H$ if it is the image of $n^{1/5}$ or more contracted edges, and blue otherwise. By the minimality of $G$ assumed above, we see that each contracted vertex of $H$ is the image under contractions of some set of vertices all of which have degree $3$ in $G$; hence, no contracted vertex is chord-adjacent in $G$ to any vertex in $P$, and no two contracted vertices are chord-adjacent.

Write $M$ for the number of minimal chords of $H$, and let $r$ denote the number of red vertices in $H$. Note that, by definition, we have $r \le n^{4/5}$ since each red vertex corresponds to a set of at least $n^{1/5}$ vertices of $G$, and these sets are all pairwise disjoint. Next, since $H$ does not contain any interlacing pairs of chords, the minimal domains of the minimal chords of $H$ are all pairwise disjoint, so if $M \ge n^{1/2}$, then one of these minimal domains contains at most $n^{1/2}$ vertices in $H$, and therefore in $G'$ and $G$ as well, in which case we are done. Therefore, we may suppose that $M \le n^{1/2}$. 

We now apply Lemma~\ref{manypar} to $H$ with $\DD$ as its designated Hamiltonian cycle to get a set $S$ of vertices such that $S$ dominates $E(H)\setminus E(\DD)$,  contains no red vertices, and contains at most $r+M-2$ pairs of consecutive vertices of $\DD$ with none of these pairs containing a blue vertex. Let us now add back to $H$ the chords that we deleted earlier, namely, those chords incident to some vertex in $P$; we call the resulting graph $H'$. Note that $X=P\cup S$ dominates the $V(H') \setminus X$ in the graph spanned by the chords of $H'$ since every vertex of degree $2$ in $H$ is chord-adjacent to some vertex in $P$; furthermore, $X$ contains at most $8m+r+M-2$ consecutive pairs of vertices of $\DD$. 

We would like to apply Lemma~\ref{lollipop} to $H'$; to do so, we need to ensure that $X$ is independent in the graph spanned by the edges of $\DD$. To ensure this, we shall contract every edge of $\DD$ between two vertices of $X$; we call the resulting graph $F$ and let $\EE$ be its designated Hamiltonian cycle obtained from $\DD$ after these contractions. Clearly, the image of $X$ in $F$ is a set that satisfies all the conditions of Lemma~\ref{lollipop} with respect to $F$ and $\EE$; therefore, it follows from Lemma~\ref{lollipop} that $F$ contains another Hamiltonian cycle $\FF$. Note that we have not contracted any edge incident to some red vertex in $H'$ in constructing $F$; moreover, we have contracted at most $8m$ blue vertices of $H'$ in constructing $F$.

Now, this cycle $\FF$ in $F$ gives rise to a cycle $\DD'$ in $H'$ missing at most $8m+r+M-2$ vertices of $H'$; indeed, at most $8m$ of the missing vertices are blue, no red vertex is missed, and the remaining missing vertices are non-contracted vertices of $G$. Now, we lift this cycle $\DD'$ in $H'$ to a cycle $\CC'$ in $G$ by replacing each red or blue vertex in $\DD'$ with an appropriate path of the original vertices of $G$; we can always choose this path to contain all the pre-images of the coloured vertex in question since, as mentioned earlier, all such vertices have degree $3$ in $H'$. It then follows that $\CC'$ misses at most $8mn^{1/5}+r+M-2$ vertices of $G$. Also, note that $\CC'\not=\CC$ since $\FF$ contains at least one chord of $F$ (and also $G$), and this chord is present $\CC'$. Therefore, $\CC'$ is a nontrivial cycle of $G$ and its length is at least 
\[n- (8mn^{1/5} +r+M-2);\]
the result follows since we know that $m \le n^{3/5}$, $r \le n^{4/5}$ and $M \le n^{1/2}$.
\end{proof}

\section{Conclusion}\label{s:conc}
Our results raise a number of questions. Perhaps the most fundamental of these concerns the nature of the error term in Theorem~\ref{mainthm}. We expect that it should be possible to improve the exponent of $4/5$ in the error term in our main result using the methods developed here, possibly up to an exponent of $1/2$; however, we chose to keep the presentation simple because we expect much more to be true.
\begin{conjecture}\label{c2}
If an $n$-vertex graph $G$ with $\delta(G) \ge 3$ contains a Hamiltonian cycle, then $G$ contains another cycle of length at least $n-K$, where $K >0$ is an absolute constant.
\end{conjecture}
It is not impossible that Conjecture~\ref{c2} holds with $K = 2$; however, we remark that the ideas developed by Thomassen~\citep{thomassen4} to disprove certain conjectures of Faudree and Schelp about path lengths in Hamiltonian graphs may be relevant in ruling out such small values of $K$.

Next, while a minimum degree of $3$ is not sufficient, as discussed earlier, to guarantee a second Hamiltonian cycle in a Hamiltonian graph, we remind the reader that it is still unknown if a minimum degree of $100$, say, suffices instead; see~\citep{deg41, deg42, deg43} for more details. 

In closing, let us mention a conjecture due to Verstra\"ete~\citep{linear} that seems closely related to the problem addressed here.
\begin{conjecture}\label{linconj}
If an $n$-vertex graph $G$ with $\delta(G) \ge 3$ contains a Hamiltonian cycle, then $G$ contains cycles of $\Omega(n)$ distinct lengths.
\end{conjecture}
It is easy to deduce a lower bound of the form $\Omega(\sqrt{n})$ for the above problem using the poset-based arguments developed here; it would be of considerable interest to push things further.

\bibliographystyle{amsplain}
\bibliography{long_h_cycles}

\end{document}